\documentclass[12pt, twoside, leqno]{article}

\usepackage{amsmath,amsthm}
\usepackage{amssymb}
\usepackage{hyperref}

\usepackage{enumerate}

\usepackage{graphicx}

\newtheorem{thm}{Theorem}[section]

\newtheorem{lem}[thm]{Lemma}
\newtheorem{prop}[thm]{Proposition}

\newtheorem{conj}[thm]{Conjecture}

\theoremstyle{definition}

\newtheorem*{xrem}{Remark}

\numberwithin{equation}{section}

\begin{document}


\baselineskip=17pt


\title{Je{\'s}manowicz' conjecture for polynomials}

\author{Jerome T. Dimabayao\\
Institute of Mathematics\\ 
University of the Philippines Diliman\\
C.P. Garcia St., U.P.Campus \\
Diliman, 1101 Quezon City, Philippines \\
Email: jdimabayao@math.upd.edu.ph}

\date{}

\maketitle


\renewcommand{\thefootnote}{}



\renewcommand{\thefootnote}{\arabic{footnote}}
\setcounter{footnote}{0}

\begin{abstract}
Let $(a,b,c)$ be pairwise relatively prime integers 
such that $a^2 + b^2 = c^2$ . 
In 1956, Je{\'s}manowicz 
conjectured that the only 
solution of $a^x + b^y = c^z$ 
in positive integers is 
$(x,y,z)=(2,2,2)$. In this note we 
prove a polynomial analogue of this 
conjecture.  
\end{abstract}

\section{Introduction}

Let $(a,b,c)$ be a Pythagorean triple, so that 
$a^2 + b^2 = c^2$. It is clear that the 
Diophantine equation 
\begin{equation}\label{Pythagorean}
a^x + b^y = c^z 
\end{equation}
has the positive integer solution 
$(x,y,z)=(2,2,2)$. In 1955/56, 
Je{\'s}manowicz \cite{Jesmanowicz} 
formulated the following conjecture: 

\begin{conj}\label{Jesma}
Let $(a,b,c)$ be a Pythagorean triple. Then 
the only positive integer solution to 
equation (\ref{Pythagorean}) is $(x,y,z)=(2,2,2)$. 
\end{conj}

Many special cases of Conjecture (\ref{Jesma}) 
have been settled for primitive Pythagorean 
triples (cf.\ e.g.\ \cite{Demjanenko, Ko1, Ko2, Le, Lu, Pod, Sierpinski}) and 
recent years saw increased activity towards 
the resolution of Conjecture (\ref{Jesma}) 
\cite{Deng-Huang,  Fuj-Miya, Ma-Chen, Miyazaki1, MYW, Tang-Weng, Tang-Yang, Terai}.  

In this note, we provide an  
analogue of Conjecture \ref{Jesma} for 
polynomials over a field $K$ of characteristic $0$. 
It is known (c.f.\ e.g.\ \cite{Ku}) that 
a triple $(a,b,c)$ of polynomials 
over such fields $K$ satisfies $a^2 + b^2 = c^2$ 
if and only if 
\[a = w(f^2-g^2), \quad b = 2wfg, \quad 
\mbox{ and } c = w(f^2 + g^2) \] 
or 
\[a = 2wfg, \quad b = w(f^2-g^2), \quad 
\mbox{ and } c = w(f^2 + g^2), \] 
where $w,f,g$ are polynomials in $K[t]$. 
If $f$ and $g$ are relatively prime 
polynomials in $K[t]$ then we call 
the triple $(f^2-g^2, 2fg, f^2 + g^2)$ 
a \emph{primitive Pythagorean triple}.
Note that if any one of the polynomials 
$f^2-g^2$, $fg$ or $f^2 + g^2$ are constant 
polynomials, then $f$ and $g$ are both constants. 
 
It is clear that the polynomial 
Diophantine equation 
\begin{equation}\label{Pythagorean poly}
(w(f^2-g^2))^x + (w(2fg))^y = (w(f^2 + g^2))^z 
\end{equation}
has the positive integer solution 
$(x,y,z)=(2,2,2)$.
We are interested in the determination of 
the complete set of solutions $(x,y,z)$ 
in positive integers of the equation above. 
We prove the following result.

\begin{thm}\label{Jesma_poly}
Let $K$ be a field with characteristic $0$. 
Let $(f^2-g^2,2fg,f^2 + g^2)$ be a primitive 
Pythagorean triple where $f$ and $g$ are 
nonconstant relatively prime polynomials 
over $K$. 
Suppose $w$ is a nonzero polynomial.
Then the only positive integer solutions to 
equation (\ref{Pythagorean poly}) are 
$(x,y,z)=(2,2,2)$ and $(x,y,z)=(2,1,1)$.
The latter occurs if and only if 
$w$ is constant, $\sqrt{w} \in K$ and $f = -g \pm 1/\sqrt{w}$. 
\end{thm}

By embedding $K$ in an algebraic closure, 
we may assume without loss of generality 
that the field $K$ is algebraically closed. 
We follow this assumption throughout the paper. 

\begin{xrem} 
Note that there are obstructions in the case 
where $K$ has positive characteristic. 
For instance, if $K$ has characteristic $2$ 
then $A^2 + B^2 = (A+B)^2$, for any 
polynomial $A,B \in K[t]$. Then 
the Diophantine equation $A^x + B^y = (A+B)^z$ 
has infinitely many solutions $(x,y,z)$ given 
by $x=y=z=2^m$ where $m$ is a nonnegative 
integer.
\end{xrem}

Let $(A,B,C)=(f^2-g^2,2fg, f^2+g^2)$ be a primitive 
Pythagorean triple in $K[t]$.
Let $(x,y,z)$ be a solution in positive 
integers to equation (\ref{Pythagorean poly}). 
We first record the following observation.

\begin{prop}\label{two are two}
If $x=y=2$ or $x=z=2$ or $y=z=2$, then we obtain 
the trivial solution $(x,y,z)=(2,2,2)$. 
\end{prop}
\begin{proof}
This follows immediately 
from unique factorization on $K[t]$. 
\end{proof}

For a polynomial $p(t) \in K[t]$ 
we let $\delta(p)$ and $\eta(p)$ denote 
the degree and the number of distinct roots 
of $p$, respectively. We prove the claimed result 
by determining bounds on $x$, $y$, and $z$. 
The following result (see for instance \cite{Snyder})
will help us achieve our aim.

\begin{thm}{\rm (Mason-Stothers)}\label{Mason}
Let $a(t)$, $b(t)$ and $c(t)$ be polynomials 
whose coefficients belong to an algebraically 
closed field $K$ with characteristic $0$. 
Suppose $a(t)$, $b(t)$ and $c(t)$ are 
not all constant, relatively prime and 
that $a(t) + b(t) = c(t)$. 
Then 
\[ \max \{ \delta(a), \delta(b), \delta(c)\} 
\leq \eta(abc) - 1. \]
\end{thm}

We split the proof of our main result into two 
parts. In section \ref{Coprime}, we prove 
Theorem \ref{Jesma_poly} in the case where 
$w$ is constant. The proof for the general 
case is given in section \ref{Noncoprime}.

\section{The primitive case}\label{Coprime}

Put $A = f^2 - g^2$, $B= 2fg$ and $C= f^2 + g^2$.
Without loss of generality, assume that 
$\delta(f) \geq \delta(g)$. 
Then note that 
$\delta(A), \delta(B),\delta(C) \leq 2 \delta(f)$. 
Assume for the moment that $w$ is a nonzero constant polynomial.
Suppose equation (\ref{Pythagorean poly}) holds 
with positive integers $x, y$, and $z$. 

Applying Theorem \ref{Mason} to 
\[ a = (wA)^x \quad 
b = (wB)^y \quad c = (wC)^z, \] 
we have 
\begin{align}
x \delta(A) &\leq
\delta(A) + \delta(B) + \delta(C) -1, \label{x-bound} \\
y \delta(B) &\leq
\delta(A) + \delta(B) + \delta(C) -1, \label{y-bound} \\ 
z \delta(C) &\leq
\delta(A) + \delta(B) + \delta(C) -1. \label{z-bound}  
\end{align}

For a polynomial $p(t) \in K[t]$, 
let $\mathrm{LT}(p)$ denote the leading term of 
$p$. We consider three cases depending on the 
relation of $\mathrm{LT}(f^2)$ with $\mathrm{LT}(g^2)$.

\begin{prop}\label{bounds} 
Let $(A,B,C)=(f^2-g^2, 2fg, f^2 + g^2)$ be a 
primitive Pythagorean triple and $w$ be a nonzero constant.
Let $(x,y,z)$ be a positive integer 
solution to equation (\ref{Pythagorean poly}). 
\begin{enumerate}
\item[(i)] If $\mathrm{LT}(f^2) \neq \pm \mathrm{LT}(g^2)$ 
then $x,z \leq 2$. 
\item[(ii)] If $\mathrm{LT}(f^2) = - \mathrm{LT}(g^2)$ 
then $x,y \leq 2$.
\item[(iii)] If $\mathrm{LT}(f^2) = \mathrm{LT}(g^2)$ 
then $y,z \leq 2$.
\end{enumerate}

\end{prop}

\begin{proof} 
We prove this by considering 
each case separately.

\underline{Case 1:} Suppose $\mathrm{LT}(f^2) \neq \pm \mathrm{LT}(g^2)$.
Since $\delta(f) \geq \delta(g)$, we have  
$\delta(A)=\delta(C)=2 \delta(f)$ and 
$\delta(B) = \delta(f) + \delta(g) \leq 2 \delta(f) = \delta(A)$.
Then relations (\ref{x-bound}) and 
(\ref{z-bound}) become 
\begin{align*}
x \delta(A) &\leq 3\delta(A) -1,  \\
z \delta(C) &\leq 3 \delta(C) -1.   
\end{align*}
Thus, $x,z \leq 2$.

\underline{Case 2:} 
Suppose $\mathrm{LT}(f^2) = - \mathrm{LT}(g^2)$. 
Then $\delta(f) = \delta(g)$, 
$\mathrm{LT}(f^2-g^2) = 2\mathrm{LT}(f^2)$. 
Thus we have
\[ 
\delta(C) < 2 \delta(f)
\quad \mbox{ and } \quad 
\delta(A) = 2\delta(f) = \delta(B).
\]  
So relations (\ref{x-bound}) and (\ref{y-bound}) 
give
\begin{align*}
x \delta(A) &< 3 \delta(A) - 1, \\
y \delta(B) &< 3 \delta(B) - 1. 
\end{align*}
Hence $x,y \leq 2$.

\underline{Case 3:} 
Suppose $\mathrm{LT}(f^2) = \mathrm{LT}(g^2)$. 
Then $\delta(f) = \delta(g)$, 
$\mathrm{LT}(f^2 + g^2) = 2\mathrm{LT}(f^2) \neq 0$. 
Thus we have
\[ \delta(A)< 2\delta(f), \quad \mbox{ and } \quad
\delta(B) =2\delta(f) = \delta(C), 
\] 
So relations (\ref{y-bound}) and (\ref{z-bound}) 
give
\begin{align*}
y \delta(B) &< 3 \delta(B) - 1, \\
z \delta(C) &< 3 \delta(C) - 1. 
\end{align*}
Hence $y,z \leq 2$.

\end{proof}

\begin{prop}\label{two are one} 
Let $(A,B,C)=(f^2-g^2,2fg, f^2+g^2)$ be a primitive 
Pythagorean triple in $K[t]$. 
Assume that $w$ is a nonzero polynomial. Then 
\begin{enumerate}
\item[(i)] $(x,y,z)=(1,r,1)$ 
is not a solution to equation (\ref{Pythagorean poly}) 
for any $r \in \mathbb{N}$;
\item[(ii)] $(x,y,z)=(r,1,1)$ 
is not a solution to equation (\ref{Pythagorean poly}) 
for any natural number $r \geq 3$;
\item[(iii)] $(x,y,z)=(2,1,1)$ 
is a solution to equation (\ref{Pythagorean poly}) 
if and only if $w$ is constant and 
$f + g = \pm 1/\sqrt{w}$.
\item[(iv)] $(x,y,z)=(1,1,r)$ 
is not a solution to equation (\ref{Pythagorean poly}) 
for any $r \in \mathbb{N}$;
\end{enumerate}
\end{prop}
\begin{proof}
If $(x,y,z)=(1,r,1)$ is a solution then 
from equation (\ref{Pythagorean poly}) 
we have 
\[ (2wfg)^r = 2wg^2. \] 
If $r=1$, then $f=g$, a contradiction. 
The inequality $r > 2$ is impossible by 
degree comparison.  
If $r = 2$ then $w$ and 
$f$ must be constant. 
Since $\delta(f) \geq \delta(g)$, 
$g$ must be constant as well.  
But this is contrary to our hypothesis.
This proves (i).

Let $r \geq 2$. With $(x,y,z)=(r,1,1)$ in equation 
(\ref{Pythagorean poly}), we have 
\[ w^{r-1}(f^2 - g^2)^r = (f-g)^2 .  \]
Since $f$ and $g$ are relatively prime, 
the above equation is equivalent to 
\begin{equation}\label{x=r,y=z=1}
w^{r-1}(f-g)^{r-2}(f+g)^r = 1.
\end{equation} 
If $r=2$ then $w$ is constant and 
equation (\ref{x=r,y=z=1}) is equivalent 
to $f+g = \pm \dfrac{1}{\sqrt{w}}$. 
This verifies (iii).
If $r\geq 3$, then equation (\ref{x=r,y=z=1}) 
implies that $f-g$ and $f+g$ are both constants. 
Hence $f$ and $g$ are both constants. 
This contradiction proves (ii).

Now assume that $(x,y,z)=(1,1,r)$, with $r \geq 2$, 
is a solution to equation (\ref{Pythagorean poly}). 
We have 
\begin{equation*}\label{x=y=1,z=r}
f^2 - g^2 + 2fg = w^{r-1}(f^2 + g^2)^r = w^{r-1}(f+gi)^r(f-gi)^r, 
\end{equation*}
where $i$ is a square root of $-1$ in $K$.
Since $f^2 - g^2 = (f \pm gi)^2 \mp 2fgi$, the 
above equation can be expressed as 
\[ (f \pm gi)^2 (1-w^{r-1}(f \pm gi)^{r-2}(f \mp gi)^r) = 2fg(\pm i-1). \]
Note that $f \pm gi$ is relatively prime to $fg$. 
Thus we see from the equation above that 
$f+gi$ and $f-gi$ are both constants. 
Therefore $f$ and $g$ are both constants. But this 
is absurd. This completes the proof of (iv) 
and the proposition.
\end{proof}

We now prove the main result in the 
case where $w$ is constant:

\begin{prop}\label{coprime-case}
Let $K$ be a field with characteristic $0$. 
Let $(f^2-g^2,2fg,f^2 + g^2)$ be a primitive 
Pythagorean triple where $f$ and $g$ are 
nonconstant relatively prime polynomials 
over $K$. 
Suppose $w$ is a nonzero constant.
Then the only positive integer solutions to 
equation \ref{Pythagorean poly} are 
$(x,y,z)=(2,2,2)$ and $(x,y,z)=(2,1,1)$.
The latter occurs if and only if 
$f = -g \pm 1/\sqrt{w}$. 
\end{prop}

\begin{proof}
If $\mathrm{LT}(f^2) \neq \mathrm{LT}(g^2)$, 
then $x,z \leq 2$ or $x,y \leq 2$ by 
Proposition \ref{bounds}-(i) and (ii). 
But items (i), (ii) and (iv) of 
Proposition \ref{two are one} imply that $x=z=2$ or $x=y=2$. 
Proposition \ref{two are two} 
then gives the solution $(x,y,z)=(2,2,2)$.
On the other hand, when 
$\mathrm{LT}(f^2) = \pm \mathrm{LT}(g^2)$, 
Proposition \ref{bounds}-(i) 
implies that we have $y,z \leq 2$. Then 
Propositions \ref{two are two} 
and \ref{two are one} imply that 
$(x,y,z)=(2,2,2)$ or $(2,1,1)$; 
and the latter holds precisely when 
$f=-g \pm 1/\sqrt{w}$. This completes the proof 
of the Proposition. 
\end{proof}

\section{The non-primitive case}\label{Noncoprime}

In this section, we treat the case where 
$w$ is a nonconstant polynomial. 
We begin with the following variant of 
relations (\ref{x-bound}), (\ref{y-bound}) 
and (\ref{z-bound}). 

\begin{prop}\label{non-coprime1}
Let $(A,B,C)=(f^2-g^2, 2fg, f^2 + g^2)$ be a 
primitive Pythagorean triple and 
$w$ be a nonconstant polynomial.
Let $(x,y,z)$ be a positive integer solution 
to equation (\ref{Pythagorean poly}). 
\begin{enumerate}
\item If $x < m := \min \{y,z\}$ then 
\begin{align}
x \delta(A) &< \delta(A) + \delta(B) + 
\delta(C) + m \delta(w) - 1, \label{x-bound_noncoprime1} \\
y \delta(B) &\leq \delta(A) + \delta(B) + 
\delta(C) - 1, \label{y-bound_noncoprime1} \\
z \delta(C) &\leq \delta(A) + \delta(B) + 
\delta(C) - 1. \label{z-bound_noncoprime1} 
\end{align}

\item If $y < m := \min \{x,z\}$ then 
\begin{align}
x \delta(A) &\leq \delta(A) + \delta(B) + 
\delta(C) - 1, \label{x-bound_noncoprime2} \\
y \delta(B) &< \delta(A) + \delta(B) + 
\delta(C) + m \delta(w) - 1, \label{y-bound_noncoprime2} \\
z \delta(C) &\leq \delta(A) + \delta(B) + 
\delta(C) - 1. \label{z-bound_noncoprime2} 
\end{align}

\item If $z < m := \min \{x,y\}$ then 
\begin{align}
x \delta(A) &\leq \delta(A) + \delta(B) + 
\delta(C) - 1, \label{x-bound_noncoprime3} \\
y \delta(B) &\leq \delta(A) + \delta(B) + 
\delta(C) - 1, \label{y-bound_noncoprime3} \\
z \delta(C) &< \delta(A) + \delta(B) + 
\delta(C) + m \delta(w) - 1. \label{z-bound_noncoprime3} 
\end{align} 
\end{enumerate}
\end{prop}

\begin{proof} 
We only give the proof of the first set of 
inequalities (\ref{x-bound_noncoprime1}) - (\ref{z-bound_noncoprime1}) 
as the rest can be verified in 
exactly the same manner. 
Without loss of generality, assume $x < y \leq z$. 
Then \begin{equation}\label{noncoprime_eqtn1}
A^x = w^{y-x} (w^{z-y} C^z - B^y). 
\end{equation} 
Any zero of $w$ is a zero of $A$. Thus $w$ is coprime 
to $B$ and to $C$. 
If $y < z$ then $w$ and  $w^{z-y} C^z - B^y$ 
are coprime. Then there exist 
coprime polynomials $A_1$ and $A_2$ such 
that 
\[ (A_1A_2)^x = A^x, \quad  
A_1^x = w^{y-x}, \quad  \mbox{ and } 
\quad   
A_2^x = w^{z-y} C^z - B^y. \] 
The first and second equations indicate that   
$\eta(wA_2BC) = \eta(ABC)$. 
The third equation consists of pairwise 
coprime terms. Theorem \ref{Mason} applied 
to the third equation gives  
\[ \max\{ \delta(A_2^x), \delta(B^y), 
\delta(w^{z-y} C^z)\} \leq \eta(ABC)-1. \]
The inequalities (\ref{y-bound_noncoprime1}) 
and (\ref{z-bound_noncoprime1}) can 
be seen immediately. Using $\delta(A_2^x)$, 
we have 
\[  x \delta(A) - (y-x)\delta(w) =
x(\delta(A) -\delta(A_1)) =  x \delta(A_2) \leq \delta(A) + \delta(B) + \delta(C) - 1, \] 
or $x (\delta(A) + \delta(w)) \leq 
\delta(A) + \delta(B) + \delta(C) + y \delta(w)- 1$. 
Since $\delta(w) > 0$, 
we obtain inequality (\ref{x-bound_noncoprime1}).

Now suppose $x < y = z$. 
Since $w$ divides $A^x$, we have $\eta(wABC) = \eta(ABC)$. 
Equation (\ref{noncoprime_eqtn1}) can be written as  
\begin{equation}\label{noncoprime_eqtn2}
\frac{(wA)^x}{w^y} = C^y - B^y. 
\end{equation} 
Note that $w$ is coprime to $B$ and to $C$. 
Applying Theorem (\ref{Mason}) to 
equation (\ref{noncoprime_eqtn2}) gives 
the desired inequalities.
\end{proof}

\begin{lem}\label{1-2-3}
Let $(A,B,C) = (f^2 - g^2, 2fg, f^2 + g^2)$ 
or $(B,A,C) = (f^2 - g^2, 2fg, f^2 + g^2)$
be a primitive Pythagorean triple. 
Let $w$ be a nonzero polynomial.  
If $A$ is nonconstant of 
even degree then
\[
wA + (wB)^2 \neq (wC)^3. \]
If $C$ is nonconstant of 
even degree then
\[ (wA)^2 + (wB)^3 \neq wC.  \]
\end{lem}

\begin{proof}
We will only prove the case where 
$(A,B,C) = (f^2 - g^2, 2fg, f^2 + g^2)$ as 
the proof for the other case is similar. 

Assume on the contrary that 
\begin{equation}\label{eqn1-2-3}
wA + (wB)^2 = (wC)^3.
\end{equation} 
Then $w$ divides $A$, say 
$A = wA_1$.  
Combining with the hypothesis 
that $(wA, wB, wC)$ is a Pythagorean triple, we have 
\[ A(wA-1) =  wC^2 (1- wC). \]
Since $wA$ is coprime to $C$, we find that $C^2$ divides $1 - wA$ 
and $A$ divides $w(wC-1)$. Let $d$ be the greatest common 
divisor of $wA-1$ and $w(1-wC)$. 
Note that $d$ and $w$ are coprime. 
Then we can write  
\begin{equation}\label{divisors1} 
wA - 1 = C^2 d \qquad \mbox{ and } 
\qquad w(1-wC) = Ad. 
\end{equation}
Then 
\[ (C^2d + 1)^2 + (wB)^2 = (1 - A_1 d)^2 \]
or 
\begin{equation}\label{divisors2}
C^4 d^2 + (wB)^2 - (A_1 d)^2 = -2(C^2 + A_1)d. 
\end{equation}
We claim that $d$ is a square. 
If $d$ is constant, then $wA-1$ and $w(1-wC)$ 
are coprime. In this case, we may assume 
without loss of generality that $d = 1$.  
Suppose $d$ is nonconstant and 
let $\alpha$ be a zero of $d$ of odd 
multiplicity $r$. 
Since $d$ divides $C-A = 2g^2$ then 
$g$ is nonconstant and $\alpha$ is 
a zero of $g$. We see that $\alpha$ is a 
zero of $B^2$ of multiplicity at least $r+1$. 
Since $(t-\alpha)^{2r}$ divides $d^2$ 
and $(t - \alpha)^{r+1}$ does not divide $d$, 
equation (\ref{divisors2}) implies that 
$t-\alpha$ divides $C^2 + A_1$. 
So $\alpha$ is a zero of 
$w(C^2 + A_1) + C-A = C(wC+1)$. As 
$d$ and $C$ are coprime, $\alpha$ is a zero 
of $wC+1$. But $\alpha$ is also a zero of 
$wC-1$, a contradiction. This proves our 
claim that $d$ is a square.  

Since $w$ and $1-wC$ are coprime and 
$\delta(A)$ is even, the second equation in 
(\ref{divisors1}) implies that 
$\delta(w)$ is even. 
Hence, $wA = h^2$ for some 
nonconstant polynomial $h$. 
Write $d = s^2$ for some polynomial $s$. 
The first equation in (\ref{divisors1}) 
implies that $h+Cs$ and $h-Cs$ are 
both constant. But these imply 
that $C$, and hence $f+ig$ and $f -ig$ are 
all constant. Consequently, $f$ and $g$ are 
constant polynomials, which is absurd.
Therefore, $wA + (wB)^2 \neq (wC)^3$. 

The same approach as above 
allows us to prove that if 
$C$ is nonconstant of even degree, 
then $(wA)^2 + (wB)^3 \neq wC$. 
This completes the proof of the lemma.
\end{proof}

\begin{lem}\label{3-1-2}
Let $(A,B,C) = (f^2 - g^2, 2fg, f^2 + g^2)$ 
or $(B,A,C) = (f^2 - g^2, 2fg, f^2 + g^2)$
be a primitive Pythagorean triple. 
Let $w$ be a nonconstant polynomial.  
Suppose that 
$\delta(A) \leq \delta(B) = \delta(C)$.  
Then 
\[ (wA)^3 + wB \neq (wC)^2. \]
\end{lem}

\begin{proof}
We only prove the case where 
$(A,B,C) = (f^2 - g^2, 2fg, f^2 + g^2)$ as 
the proof for the other case is similar. 
Without loss of generality, assume  
$\delta(f) \geq \delta(g)$. So 
$\delta(f) \leq \delta(A) \leq \delta(B) = \delta(C) = 2 \delta(f)$. 

Suppose that 
\begin{equation}\label{hyp3-1-2}
(wA)^3 + wB = (wC)^2.
\end{equation}  
Then $w$ divides $B$, say $B = wB_1$.  
Since $(wA, wB, wC)$ is a Pythagorean triple, we have 
\[ A^2(wA-1) =  B_1(wB-1). \]
By comparing degrees, we see that  
\[ \delta(w) + 3 \delta(A) = 4 \delta(f). \] 
Thus, $\delta(w) \leq \delta(f)$.

Since $A$ is coprime to $B$, we find that 
$A^2$ divides $wB - 1$ 
and $B_1$ divides $wA-1$. Let $d$ be the greatest common 
divisor of $wA-1$ and $wB-1$. 
Note that $d$ and $w$ are coprime. 
Then we can write  
\begin{equation}\label{divisors1.1} 
wA - 1 = B_1 d \qquad \mbox{ and } 
\qquad wB - 1 = A^2 d. 
\end{equation}
In addition, we have 
\begin{equation}\label{divisors2.1}
w^2AB - 1 = C^2 d.
\end{equation}
From (\ref{divisors1.1}), we see that 
$A$ divides $wB + B_1d$. Since 
$A$ is coprime to $B$, this  
implies that $A$ divides $w^2 + d$.
Since $B_1$ is coprime to $C$, 
equations (\ref{hyp3-1-2}) 
and (\ref{divisors2.1}) imply that 
$B_1$ divides $w^2 + d$. 
Hence, $AB$ divides $w^3 + dw$.
Consequently, 
\[ 
3 \delta(f) \leq
\delta(A) + \delta(B) =
\delta(AB) \leq 3 \delta(w). 
\]
Therefore $\delta(w) = \delta(f)$ and 
the above chain of inequalities is 
a chain of equalities. In particular, 
we have $\delta(A) = \delta(f)$. 
Since $A = f^2 - g^2$, 
either $f-g$ or $f+g$ is a constant. 
Without loss of generality, suppose 
$f+g$ is constant. Equation (\ref{hyp3-1-2}) 
implies that 
\[ 2wfg \equiv 1 \pmod{(f-g)^2}. \]
Applying the identity 
$(f+g)^2 - (f-g)^2 = 4fg$ gives 
\[ w(f+g)^2 \equiv 2 \pmod{(f-g)^2}. \] 
Since $f+g$ is constant, we see that 
\[ 2 \delta(f) = 2 \delta(f-g) \leq \delta(w) = \delta(f). \]
This means that $f$, and hence $g$, must be constants. 
Contradiction.  
\end{proof}

\begin{prop}\label{noncoprime_exponents are equal}
Let $(A,B,C)=(f^2-g^2, 2fg, f^2 + g^2)$ be a 
primitive Pythagorean triple and 
$w$ be a nonconstant polynomial.
Let $(x,y,z)$ be a positive integer solution 
to equation (\ref{Pythagorean poly}). 
Then $x = y = z$.
\end{prop}

\begin{proof}
To prove the proposition, we show that 
if $(x,y,z)$ is a positive integer solution 
to equation (\ref{Pythagorean poly}), then  
we must have $x \geq \min \{y,z\}$, 
$y \geq \min \{x,z\}$ and 
$z \geq \min \{x,y\}$, all at the same time.  

Without loss of generality, we suppose that 
$\delta(f) \geq \delta(g)$. 
As we saw in the last section, the proof  
for each case is split into three 
subcases depending on the relationship between 
the leading terms of $f^2$ and $g^2$ and its  
consequences on the relationship between 
$\delta(A)$, $\delta(B)$ and $\delta(C)$. 
For reference, we list them below:

(1) $\mathrm{LT}(f^2) \neq \pm \mathrm{LT}(g^2)$, so 
that $\delta(f) \leq \delta(B) \leq \delta(A) = \delta(C) = 2 \delta(f)$; 
 
(2) $\mathrm{LT}(f^2) = \mathrm{LT}(g^2)$, so 
that $\delta(f) \leq \delta(A) < \delta(B) = \delta(C) = 2 \delta(f)$; 

(3) $\mathrm{LT}(f^2) = -\mathrm{LT}(g^2)$, so 
that $\delta(f) \leq \delta(C) < \delta(A) = \delta(B) = 2 \delta(f)$.  

We begin by showing that 
$x \geq \min \{ y,z \}$.
By way of contradiction, suppose $x < \min \{ y,z \}$. 
Consider the following cases.

\underline{Case 1.1:} 
Suppose $\mathrm{LT}(f^2) \neq \pm \mathrm{LT}(g^2)$.
Then $z \leq 2$ by relation (\ref{z-bound_noncoprime1}). 
So $x = 1$ and $z = 2$. 
If $y \leq z$, then we must have 
$y=z=2$ and $x = 1$. But Proposition \ref{two are two} 
implies that this is absurd. 
If $y > z =2$, then comparing 
degrees of both sides of the equation $wA + (wB)^y = (wC)^2$, 
we obtain  
\[ y (\delta(w) + \delta(f)) \leq y (\delta(w) + \delta(f) + \delta(g)) 
= 2 (\delta(w) + 2\delta(f)). \]
Since $w$ is nonconstant we see that 
$0 < (y-2)\delta(w) \leq (4-y)\delta(f)$.
This is absurd when $y \geq 4$. 
Thus, $y=3$. But this is also impossible 
by Lemma \ref{3-1-2}.

\underline{Case 1.2:} 
Suppose $\mathrm{LT}(f^2) = \mathrm{LT}(g^2)$.
Then relations (\ref{y-bound_noncoprime1}) and 
(\ref{z-bound_noncoprime1})  
imply that $y \leq 2$ and $z \leq 2$. Since $x < y,z$, 
we must have $y=z=2$ and $x=1$. This 
contradicts Proposition \ref{two are two}. 

\underline{Case 1.3:} 
Suppose $\mathrm{LT}(f^2) = -\mathrm{LT}(g^2)$. 
Relation (\ref{y-bound_noncoprime1}) implies 
that $y \leq 2$.  So we must have $y=2$ and $x=1$.
By hypothesis and Proposition \ref{two are two}, 
we have $z > 2$.
Comparing degrees of both sides of the equation
$wA + (wB)^2 = (wC)^z$,
we see that 
\[ 2( \delta(w) + 2\delta(f)) = 
z (\delta(w) + \delta(C)) 
\geq z (\delta(w) + \delta(f)). \] 
Since $w$ is nonconstant we see that
$0 < (z-2)\delta(w) \leq (4-z) \delta(f)$. 
This is absurd when $z \geq 4$.
Since $\delta(A)$ is even, 
Lemma \ref{1-2-3} shows that the case $z = 3$ 
is also impossible.

Therefore, $x \geq \min \{ y,z \}$.

Next, we prove that $y \geq \min \{ x,z \}$.
Assume that $y < \min \{ x,z \}$.
Consider the following cases:

\underline{Case 2.1:} 
Suppose $\mathrm{LT}(f^2) \neq \pm \mathrm{LT}(g^2)$.
Relations (\ref{x-bound_noncoprime2}) 
and (\ref{z-bound_noncoprime2}) 
imply that $x \leq 2$ and $z \leq 2$, 
respectively.   
Since $0 < y < x,z$, we must have $x= 2 =z$ 
and $y=1$. This is impossible by Proposition \ref{two are two}.

\underline{Case 2.2:} 
Suppose $\mathrm{LT}(f^2) = \mathrm{LT}(g^2)$. 
Then $z \leq 2$ by relation (\ref{z-bound_noncoprime2}). 
So $y=1$ and $z=2$.  
If $x \leq z$, then we must have $x=z=2$ and $y=1$. 
This contradicts Proposition \ref{two are two}.  
If $x > z = 2$ then 
comparing degrees of both sides of the 
equation $(wA)^x + wB = (wC)^2$, we obtain 
\[ x (\delta(w) + \delta(f)) 
\leq x (\delta(w) + \delta(A)) 
= 2 (\delta(w) + 2\delta(f)). \]
Since $w$ is nonconstant we see that 
$0 < (x-2)\delta(w) \leq (4-x)\delta(f)$.
Thus, $x=3$. But Lemma \ref{3-1-2} 
implies that this is impossible.

\underline{Case 2.3:} 
Suppose $\mathrm{LT}(f^2) = -\mathrm{LT}(g^2)$. 
Then $x \leq 2$ by relation (\ref{x-bound_noncoprime2}). 
So $y=1$ and $x=2$.  
If $z \leq x$, then we must have $x=z=2$ and $y=1$. 
This contradicts Proposition \ref{two are two}.  
If $z > x = 2$ then 
comparing degrees of both sides of the equation 
$(wA)^2 + wB = (wC)^z$, we obtain  
\[ 2 (\delta(w) + 2\delta(f)) = 
z (\delta(w) + \delta(C)) \geq 
z (\delta(w) + \delta(f)). \]
Since $w$ is nonconstant we see that 
$0 < (z-2)\delta(w) \leq (4-z)\delta(f)$.
Thus, $z=3$. As $\delta(B)$ is even, 
this is impossible by Lemma \ref{1-2-3}. 

Therefore, $y \geq \min \{ x,z \}$

Finally, we show that $z \geq \min \{ x,y \}$.
Assume that $z < \min \{ x,y \}$.
Consider the following cases:

\underline{Case 3.1:} 
Suppose $\mathrm{LT}(f^2) \neq \pm \mathrm{LT}(g^2)$.
Then $x \leq 2$ by relation (\ref{x-bound_noncoprime3}). 
So $z=1$ and $x=2$.  
If $y \leq x$, then we must have $x=y=2$ and $z=1$. 
This contradicts Proposition \ref{two are two}.  
If $y > x = 2$ then 
comparing degrees of both sides of the equation 
$(wB)^y = wC - (wA)^2$, we obtain  
\[ y (\delta(w) + \delta(f)) \leq y (\delta(w) + \delta(B)) 
= 2 (\delta(w) + 2\delta(f)). \]
Since $w$ is nonconstant we see that 
$0 < (y-2)\delta(w) \leq (4-y)\delta(f)$.
Thus, $y=3$. As $\delta(C)$ is even, 
this is impossible by Lemma \ref{1-2-3}.

\underline{Case 3.2:} 
Suppose $\mathrm{LT}(f^2) = \mathrm{LT}(g^2)$.
Then relation (\ref{y-bound_noncoprime3}) 
gives $y \leq 2$. So $z=1$ and $y=2$.  
If $x \leq y$, then we must have 
$x=y=2$ and $z=1$, contradicting Proposition \ref{two are two}. 
If $x > y = 2$ then 
comparing degrees of both sides of the equation 
$(wA)^x = wC - (wB)^2 $, we obtain  
\[ x (\delta(w) + \delta(f)) \leq x (\delta(w) + \delta(A)) 
= 2 (\delta(w) + 2\delta(f)). \]
Since $w$ is nonconstant we see that 
$0 < (x-2)\delta(w) \leq (4-x)\delta(f)$.
This is absurd when $x \geq 4$.
Thus, $x=3$. As $\delta(C)$ is even, 
this is impossible by Lemma \ref{1-2-3}.

\underline{Case 3.3:} Suppose $\mathrm{LT}(f^2) = -\mathrm{LT}(g^2)$. 
Relations (\ref{x-bound_noncoprime3}) 
and (\ref{y-bound_noncoprime3}) 
imply that $x \leq 2$ and $y \leq 2$, 
respectively.   
Since $0 < z < x$, we must have $x= 2 =y$ 
and $z=1$. This is impossible.

Therefore, $z \geq \min \{ x,y \}$. 
This finishes the proof of the Proposition.
\end{proof}

We list one more result which allows us 
to reduce the proof further into the 
simplest cases.

\begin{prop}\label{two equal exp1}
Let $(A,B,C) = (f^2 - g^2, 2fg, f^2 + g^2)$ 
be a primitive Pythagorean triple 
and $w$ be a nonzero polynomial. 
Let $m$ and $n$ be positive 
integers with $m \leq n$ such that 
\[ (wA)^n = (wC)^m - (wB)^m. \]
Then $m \leq 2$.
\end{prop}

\begin{proof}
Let $\zeta$ be a primitive $m$th root of unity. 
We have 
\[ w^{n-m}A^n = \prod_{j=1}^m (C - \zeta^j B), \]
where the factorization on the right consists 
of $m$ distinct pairwise coprime factors. 
Suppose $m \geq 3$. Then $n \geq 3$, and 
the above equation can be written as 
\[ w^{n-m}(C+B)A^{n-2} = \prod_{j=1}^{m-1} (C - \zeta^j B). \]
So $C + B$ divides $\prod_{j=1}^{m-1} (C - \zeta^j B)$. 

If $m$ is odd, then $C+B$ must be constant and $A$ 
divides the product. But any zero of $A$ 
is a zero of $C^2 - B^2$, while the product is 
coprime to $C^2 - B^2$. Contradiction. 

Suppose $m$ is even. 
Then we cancel $C+B$ out in the last equation to obtain 
\[ w^{n-m}A^{n-2} = \prod_{\substack{j=1 \\ j \neq m/2}}^{m-1} (C - \zeta^j B). \]
Thus $A$ divides 
$\prod_{\substack{j=1 \\ j \neq m/2}}^{m-1} (C - \zeta^j B)$. 
As in the previous case, this is impossible. 
This completes the proof. 
\end{proof}

We are now ready to finish the proof 
of our main result.

\begin{proof}[Proof of Theorem \ref{Jesma_poly}]
Let $w$ be a nonzero polynomial. 
The case where $w$ is constant has been settled 
by Proposition \ref{coprime-case}.
Suppose that $w$ is nonconstant and 
let $x,y,z$ be positive integers that satisfy 
equation (\ref{Pythagorean poly}). 
Proposition \ref{noncoprime_exponents are equal} implies $x= y = z$. 
Moreover, we have from Proposition \ref{two equal exp1} that $x,y,z \leq 2$. 
Suppose $x = y =z = 1$. Then $A = C - B$, or 
\[ f^2 - g^2 = (f - g)^2.  \]
This implies $f + g = f - g$ or $g =0$. 
Therefore, $x = y = z = 2$. 
The proof is complete.
\end{proof}

\end{document}